\documentclass[11pt]{article}
\usepackage{graphicx,mathtools,amssymb,bbm,breqn,amsxtra,amsthm,float,authblk,appendix,subcaption}
\usepackage[margin=1.0in]{geometry}
\usepackage{titlesec}
\usepackage{hyperref}
\usepackage[textsize=tiny]{todonotes}

\titleformat*{\section}{\normalsize\bfseries}
\titleformat*{\subsection}{\normalsize\bfseries}
\titleformat*{\subsubsection}{\normalsize\bfseries}
\titleformat*{\paragraph}{\normalsize\bfseries}
\titleformat*{\subparagraph}{\normalsize\bfseries}
\newcommand{\R}{\mathbb{R}}
\newcommand{\Q}{\mathbf{Q}}

\newcommand{\de}{\text{ d}}

\newcommand{\He}{\text{He}}
\newcommand{\re}{\text{Re}}

\newcommand{\A}{\textbf{A}}
\newcommand{\sA}{\mathcal{A}}
\newcommand{\B}{\textbf{B}}
\newcommand{\Tr}{\text{Tr}\,}
\newcommand{\n}{\mathbf{n}}
\newcommand{\m}{\mathbf{m}}

\newtheorem{lem}{Lemma}

\newtheorem{proposition}{Proposition}
\newtheorem{remark}{Remark}
\graphicspath{{Figures/}}

\title{\Large{\textbf{Computing eigenfunctions of the multidimensional Ornstein--Uhlenbeck operator}}}
\author[1]{Benjamin J.~Zhang\footnote{Corresponding author. Email: \texttt{bjz@mit.edu}}}
\author[2]{Tuhin Sahai}
\author[1]{Youssef M.~Marzouk}
\affil[1]{\small{Center for Computational Science and Engineering, Massachusetts Institute of Technology}}
\affil[2]{\small{Raytheon Technologies Research Center}}

\date{\today}

\begin{document}

\maketitle

\begin{abstract}
We discuss approaches to computing eigenfunctions of the Ornstein--Uhlenbeck (OU) operator in more than two dimensions. While the spectrum of the OU operator and theoretical properties of its eigenfunctions have been well characterized in previous research, the practical computation of general eigenfunctions has not been resolved. We review special cases for which the eigenfunctions can be expressed exactly in terms of commonly used orthogonal polynomials. Then we present a tractable approach for computing the eigenfunctions in general cases and comment on its dimension dependence. 
\end{abstract}

\section{Introduction}
The Ornstein--Uhlenbeck (OU) operator naturally arises in many fields. In stochastic differential equations (SDEs), the OU operator is the generator of the Ornstein--Uhlenbeck semigroup, which describes the evolution of statistics OU \emph{processes}, which are linear time-homogeneous SDEs \cite{pavliotis2014stochastic}. Eigenfunctions of the OU operator also appear in Koopman operator analysis of linear stochastic dynamical systems, as the stochastic Koopman operator for linear SDEs
has the same eigenfunctions as the OU operator \cite{vcrnjaric2019koopman}. These eigenfunctions have been useful in perturbation analysis of Fokker--Planck equations for nonlinear SDEs \cite{leen2011perturbation}. Recently, the eigenfunctions have been shown to be useful in constructing importance sampling schemes for rare event simulation \cite{zhang2021koopman}. The OU process is also used to model dynamical phenomena in financial mathematics \cite{vasicek1977equilibrium,nicolato2003option} and neuroscience \cite{ricciardi1979ornstein,feng2009can}.

Properties of the spectrum and eigenfunctions of the OU operator have been thoroughly explored in the literature. For example, the spectrum has been computed exactly, and many theoretical properties of the eigenfunctions---such as the fact that they are polynomials and are complete in certain weighted $L^p$ spaces---have been established~\cite{metafune2002spectrum}. There are, however, applications in which one needs to directly work with the eigenfunctions \cite{leen2011perturbation,zhang2021koopman}. The exact form of the eigenfunctions has only been recorded in limited special cases, and a comprehensive approach to computing the eigenfunctions, in general, has not been found by the authors. In this note, we describe certain cases in which the multidimensional OU eigenfunctions can be represented compactly in terms of commonly used orthogonal polynomials. Then we outline a direct way of computing them in a more general setting. This note is targeted towards those who are looking for methods to \emph{exactly} compute the eigenfunctions of the OU operator for \emph{general} diagonalizable drift and diffusion matrices, in \emph{arbitrary dimensions}.

\section{Theory and special cases}

\subsection{Notation and problem setting }
Let $\A$ and $\B$ be $d\times d$ and $d\times r$ real-valued matrices, respectively, with $d \ge r$, and define $\Q = \frac{1}{2} \B\B^\top$, where $^\top$ denotes the matrix transpose. Below, $\overline{\lambda}$ will denote the complex conjugate, $^*$ will denote the conjugate transpose, and $\langle u,v\rangle = u^* v$ will be the inner product. Assume that the eigenvalues of $\A$ have strictly negative real parts, and that none of the left eigenvectors of $\A$ are contained in the kernel of $\B^\top$. We also assume that $\A$ is diagonalizable; $\B$ may be rank-deficient.\footnote{When $\B$ is rank-deficient, this leads to the case where the Ornstein--Uhlenbeck operator is hypoelliptic \cite{metafune2002spectrum}.} We study the computation of the eigenfunctions on $L^p(\nu)$ for $p>1$, where $\nu$ is the invariant probability measure associated with the linear system of the operator. The existence of a nondegenerate invariant measure $\nu$ is guaranteed by the assumptions on $\A$ and $\B$ \cite{metafune2002spectrum}. The OU operator $\sA$ is given by 
\begin{align}
    \sA \psi = \langle \A x, \nabla \psi \rangle + \Tr \Q \nabla^2 \psi = \sum_{i = 1}^d (\A x)_i \frac{\partial \psi}{\partial x_i } + \sum_{i ,j = 1}^d \Q_{ij} \frac{\partial ^2 \psi}{\partial x_i \partial x_j}.
\end{align}
In the context of stochastic differential equations, the OU operator is the infinitesimal generator of the OU process, which is a time-homogeneous linear SDE,
\begin{align}
  \de X_t = \A X_t  \de t + \B \de W_t \, ,
\end{align}
where $W_t$ is a standard $d$-dimensional Brownian motion. 

The spectrum of the Ornstein--Uhlenbeck operator and its associated semigroup has been well studied (for example, see \cite{metafune2002spectrum,bogachev2018ornstein,lunardi1997ornstein}). Previous research has characterized the eigenfunctions of the self-adjoint OU operator, which corresponds to the case when $\A$ is self-adjoint and shares the same eigenvectors as $\B$. In this case, the eigenfunctions are the tensorized Hermite polynomials \cite{pavliotis2014stochastic}. In $d=2$ dimensions, if $\A$ has only complex eigenvalues and is normal (i.e., $\A^\top\A= \A\A^\top$), the eigenfunctions are the so-called Hermite-Laguerre-It\^o (HLI) polynomials \cite{chen2014eigenfunctions}. In general the OU operator is not self-adjoint, so we cannot appeal to the spectral theory of self-adjoint operators to prove the existence of eigenvalues. Nevertheless, the seminal work of \cite{metafune2002spectrum} shows that, under mild conditions, the OU operator has a pure point spectrum in $L^p(\nu)$ for $1<p<\infty$, where $\nu$ is the stationary measure of the OU process. Moreover, \cite{metafune2002spectrum} shows that the eigenfunctions form a complete basis in $L^p(\nu)$ for $1<p<\infty$, the eigenfunctions are all polynomials, and that the eigenvalues and eigenfunctions are the same for all $1<p<\infty$. We summarize these facts by recalling the following propositions from \cite{metafune2002spectrum}.
\begin{proposition}[{\cite[Theorem 3.1]{metafune2002spectrum}}]
	Let $-\lambda_1,\ldots,-\lambda_l$ be the distinct eigenvalues of \textbf{A}, where $\lambda_k>0$ for all $k$. Then the spectrum of $\mathcal{A}$ is given by
	\begin{align*}
		\left\{-\sum_{k = 1}^l n_k \lambda_k: n_k \in \mathbb{N} \right\}.
	\end{align*}
	Moreover, the linear span of the eigenfunctions of $\mathcal{A}$ is dense in $L^p(\nu)$. 
	\label{prop1}
\end{proposition}

\begin{proposition}[{\cite[Proposition 3.1]{metafune2002spectrum}}]
	Suppose that $u$ is in the domain of $\mathcal{A}$ and satisfies $(\gamma-\mathcal{A})u = 0$ for some $\gamma\in\mathbb{C}$. Then $u$ is a polynomial of degree less than or equal to $|\re(\gamma)/s(A)|$, where $s(A) = \sup_k\{\re(\lambda_k)\}$. That is, the eigenfunctions of the OU operator are polynomials. 
\end{proposition}

In \cite{leen2016eigenfunctions}, the authors describe the generalized form of the OU eigenfunctions in terms of ladder operators. Given a seed eigenfunction, repeated application of the ladder operators generates other eigenfunctions. While compact in its mathematical formulation, the approach is not easily amenable to practical computations. To make computing eigenfunctions tractable, we represent the OU operator as a matrix acting in 
some chosen basis of polynomials. Since it is known that the eigenfunctions of the OU operator are polynomials, an exact matrix representation of the OU operator on some finite dimensional vector space of polynomials is possible \cite{metafune2002spectrum}. 

While the pure point spectrum of the OU operator on $L^p(\nu)$ spaces with $p > 1$ is known explicitly, there is no explicit expression for the eigenfunctions in general. In \cite{metafune2002spectrum}, the authors showed that for $p>1$ the spectrum of the OU operator is the same as that of 
\begin{align}
	\mathcal{L}\psi  \coloneqq \langle x,\A^\top \nabla  \psi\rangle = \sum_{k = 1}^d x_k \left(\A^\top\nabla \psi \right)_k,
\end{align} regardless of the form of the diffusion term. In Section \ref{sec:general}, we will show how the eigenfunctions of $\mathcal{L}$ in fact comprise a judicious choice of basis for computing the eigenfunctions in general. The following lemma will be useful later when converting the OU eigenvalue problem into a matrix eigenvalue problem. 
\begin{lem}
  Let $\A\in\R^{d\times d}$ be diagonalizable and full rank. Let $f_i$ be a left eigenvector of $\A$ with eigenvalue $-\lambda_k$, i.e., $f_k^* \A = -\lambda_k f_k^*$. Let $\mathbf{n} \in \mathbb{N}_0^d$ be a $d$-dimensional multi-index of nonnegative integers. The eigenfunctions of the operator $\mathcal{L}\psi = \langle x, \A^\top \nabla \psi\rangle$ are
\begin{align}
    \psi_\n(x) := \prod_{k = 1}^d \psi_{n_k}(x) =  \prod_{k = 1}^d \langle x, f_k\rangle ^{n_k}
    \label{eq:basisfuncts}
\end{align}
with eigenvalues
\begin{align}
    \mu_{\mathbf{n}} = -\sum_{k = 1}^d n_k \lambda_k.
\end{align}
\label{lem}
\end{lem}

\begin{proof}
	Observe that
	\begin{align*}
		\left\langle x, \A^\top \nabla \prod_{k  =1}^d \psi_{n_k}(x) \right\rangle &=\left\langle  x,\A^\top\sum_{j = 1}^d n_j\langle x,f_j\rangle^{n_j-1} f_j \prod_{k \neq j}^d \langle x,f_k\rangle^{n_k} \right\rangle \\
		& = \sum_{j = 1}^d n_j \langle x,\A^\top f_j \rangle\langle x,f_j \rangle^{n_j-1} \prod_{k\neq j}^d \langle x,f_k\rangle^{n_k} \\
		& = \sum_{j = 1}^d -n_j \lambda_j \prod_{k = 1}^d \langle x,f_k\rangle ^{n_k} \\
		& = -\left(\sum_{j = 1}^d n_j \lambda_j \right)\psi_{\n}(x).
	\end{align*}
\end{proof}

\subsection{Special cases}
The eigenfunctions of $\sA$ are well-known for certain special cases. We outline some of these cases here.
\subsubsection{$\A$ and $\B$ are self-adjoint and simultaneously diagonalizable}
\label{sec:Aselfadjoint}
Here we study the case where $\A$ and $\B$ are self-adjoint and simultaneously diagonalizable. Then the eigenvalue problem is decomposable into $d$ one-dimensional eigenvalue problems, each of which is a Hermite differential equation. The relationship between the Hermite polynomials and the OU operator with $\A = \B = \mathbf{I}$ has been well-studied (see, e.g., \cite{pavliotis2014stochastic, bogachev2018ornstein, lunardi1997ornstein}  and the references therein). The extension to the present scenario is straightforward.  The eigenvalues of $\A$ are real and the eigenvectors are orthogonal.  Suppose we have $\A e_k = -\lambda_k e_k$ and $\B e_k = \sigma_k e_k$ for $\lambda_k,\sigma_k >0$, with $\langle e_j,e_k \rangle = \delta_{jk}$. We first show that univariate Hermite polynomials defined in the direction of each of the eigenvectors are eigenfunctions. That is, we make the \emph{ansatz} that
\begin{align}
	\phi_k(x) = g(\langle x,e_k\rangle)
\end{align}
and show that $g$ can be expressed in terms of a Hermite polynomial. The gradient and Hessian of this function are
\begin{align}
	\nabla \phi_k(x) = g'(\langle x,e_k \rangle )e_k, \,\,\,\, \nabla^2 \phi_k(x) = g''(\langle x,e_k\rangle ) e_ke_k^\top,
\end{align}
so the OU operator applied to $\phi_k(x)$ yields
\begin{align*}
	\mathcal{A}\phi_k(x) = \langle x,\A^\top e_k\rangle g'(\langle x,e_k\rangle) + \frac{1}{2} \Tr \left[\B\B^\top e_ke_k^\top \right] g''(\langle x,e_k\rangle).
\end{align*}
This yields the eigenvalue problem,
\begin{align}
	-\lambda_k \langle x,e_k\rangle g'(\langle x,e_k\rangle) + \frac{1}{2}\sigma^2_k   g''(\langle x,e_k\rangle) = \mu_k g(\langle x,e_k\rangle). 
\end{align}
Recall that the probabilist's Hermite polynomials $\He_n(x)$ solve the Hermite differential equation $-x \He_n'(x) + \He_n''(x) = -n \He_n(x)$. Therefore, notice that if $g(\langle x,e_k\rangle) = \He_{n_k}\left( \sqrt{\frac{2\lambda_k}{\sigma^2_k}}\langle x,e_k\rangle\right)$ for some $n_k\in \mathbb{N}_0$, we then have
\begin{align}
 	-\lambda_k \sqrt{\frac{2\lambda_k}{\sigma^2_k}}\langle x,e_k\rangle \He_{n_k}'\left(\sqrt{\frac{2\lambda_k}{\sigma^2_k}}\langle x,e_k\rangle \right) + \lambda_k \He_{n_k}''\left(\sqrt{\frac{2\lambda_k}{\sigma^2_k} } \langle x,e_k\rangle \right) = -n_k \lambda_k \He_{n_k}\left( \sqrt{\frac{2\lambda_k}{\sigma^2_k} } \langle x,e_k\rangle\right). \label{eq:hermitede}
 \end{align} 
 In the next section, we will see that products of different univariate eigenfunctions defined in the directions of the eigenvectors of $\A$ are eigenfunctions of $\mathcal{A}$. Specifically, 
 \begin{align}
 	\phi_{\mathbf{n}}(x) = \prod_{k = 1}^d \phi_{n_k}(x) = \prod_{k = 1}^d \He_{n_k}\left(\sqrt{\frac{2\lambda_k}{\sigma_k^2}} \langle x,e_k\rangle \right),
 \end{align}
 is an eigenfunction with eigenvalue $\mu_\n = -\sum_{k = 1}^d n_k \lambda_k$.

\subsubsection{$\A$ is normal and $\A, \B$ are simultaneously diagonalizable}
Now we consider the case where $\A$ is normal, i.e., $\A\A^\top = \A^\top \A$, but not necessarily self-adjoint. For this case, it is possible for $\A$ to have complex eigenvalues. In \cite{chen2014eigenfunctions}, for an OU operator with
\begin{align*}
	\A = \begin{bmatrix}
		-a & b \\ -b & -a 
	\end{bmatrix}, \text{ and } \B = \sigma\mathbf{I},
\end{align*}
the OU eigenfunctions are found to be the Hermite-Laguerre-It\^o (HLI) polynomials. The HLI polynomials are
\begin{align*}
	J_{m,n}(z,\bar{z};\rho) = \begin{cases}
		(-1)^n n! z^{m-n} L_n^{m-n}(z\bar{z},\rho), \;\;\;\; m\ge n  \\
		(-1)^m m! \bar{z}^{n-m} L_m^{n-m}(z\bar{z},\rho),\;\;\;\; m<n
	\end{cases}
\end{align*}
where $L_k^\alpha(z,\rho)$ are the generalized Laguerre polynomials, $\rho = \sigma^2/a$, and $z = x_1 + ix_2$. The OU eigenvalues in this case are $\mu_{m,n} = -(m+n)a + i(m-n)b$. 
\cite{chen2014eigenfunctions} also generalizes this result to $d$ dimensions, for even $d$, when the matrix $\A$ is normal but only has complex eigenvalues. Similar to the self-adjoint case, the eigenfunctions are simply products of the HLI polynomials on each of the eigenspaces.

We now explicitly write the eigenfunctions for general normal matrices $\A$ and for self-adjoint matrices $\B$ that share the same eigenspace as $\A$. The latter conditions imply that $\B$ is only has real eigenvalues. 
While the expression follows simply from previous results, to our knowledge no previous work has explicitly computed these eigenfunctions. 

When $\A$ has both real and complex eigenvalues, the eigenfunctions are products of Hermite and HLI polynomials. Suppose $\A$ has $l$ eigenspaces, with $l'$ real eigenspaces and $l-l'$ complex eigenspaces; that is, $\A$ has $l'$ real eigenvalues and $l-l'$ pairs of complex eigenvalues. Let $f_i$ denote a unit left eigenvector of $\A$ with eigenvalue $-\lambda_k$; that is, $f_k^* \A = -\lambda_k f_k^*$. Let the first $l'$ eigenvalues be real and the next $l-l'$ eigenvalues come in complex conjugate pairs. To be clear, for complex eigenvalues, we write $\lambda_k = a_k -ib_k$. Let $\B$ be such that $\B f_k = \sigma_kf_k$, where $\sigma_k >0$. Note that $l' + 2(l-l') = d$. Let $\n\in\mathbb{N}_0^d$ be a multi-index defined as $\n = (n_1,\ldots,n_{l'},n_{(l'+1)1},n_{(l'+1)1},\ldots,n_{l1},n_{l2})$. Then the eigenfunction of the corresponding OU operator is 
\begin{align}
    \phi_{\textbf{n}}= \prod_{k = 1}^l \phi_{n_k}(x) = \prod_{k = 1}^{l'} \He_{n_k}\left(\sqrt{\frac{2\lambda_k}{\sigma_k^2}}\langle x,f_k\rangle \right) \cdot \prod_{k = l'+1}^l J_{n_{k1},n_{k2}} \left(\sqrt{2}\langle x,f_k\rangle,\sqrt{2}\,\,\overline{\langle x,f_k\rangle}; \rho_k \right)
    \label{eq:normaleigfunc}
\end{align}
with eigenvalue $\mu_{\n} = \sum_{k = 1}^{l'} -n_k\lambda_k - \sum_{k = l'+1}^l\left[ (n_{k1}+n_{k2})a_k - i(n_{k1}-n_{k2}) b_k\right]$, and $\rho_k = \sigma_k^2/a_k.$ We show that Equation \eqref{eq:normaleigfunc} is indeed an OU eigenfunction. We first compute the following two expressions:
\begin{align*}
 \langle x,\A^\top\nabla \phi_{\n}(x) \rangle =& \sum_{k = 1}^l \langle x, \A^\top \nabla \phi_{n_k}(x) \rangle \prod_{j = 1,k\neq j}^l \phi_{n_j}(x) \\
 \frac{1}{2} \Tr \B\B^\top\nabla^2\phi_{\n} (x) =& \frac{1}{2}\sum_{k = 1}^l \Tr\B\B^\top\nabla^2 \phi_{n_k}(x) \prod_{j = 1,k\neq j}^l\phi_{n_j}(x)\\& + \sum_{k >j}^l \Tr\B\B^\top\nabla \phi_{n_k}(x)\nabla \phi_{n_j}(x)^\top \prod_{k' = 1,k'\neq k\neq j}^l \phi_{n_{k'}}(x).
 \end{align*}
The gradient of $\phi_{n_k}$ is 
\begin{align*}
	\nabla \phi_{n_k}(x) = \begin{dcases}
		\sqrt{\frac{2\lambda_k}{\sigma_k^2}}\He_{n_k}'\left(\sqrt{\frac{2\lambda_k}{\sigma^2_k}}\langle x,f_k\rangle \right)f_k \, \text{ if $1\le k\le l'$}\\
		\sqrt{2} \frac{\partial J_{n_{k1},n_{k2}}}{\partial z_k} f_k+ \sqrt{2} \frac{\partial J_{n_{k1},n_{k2}}}{\partial \overline{z}_k}\overline{f}_k \text{ if } l'+1\le k\le l,
	\end{dcases}
\end{align*}
where $z_k = \sqrt{2} \langle x,f_k\rangle$. 
The Hessian of $\phi_{n_k}$ is
\begin{align*}
 	\nabla^2\phi_{n_k} = \begin{dcases}
 		\frac{2\lambda_k}{\sigma_k^2}\He''_{n_k}\left(\sqrt{\frac{2\lambda_k}{\sigma_k^2}} \right) f_kf_k^\top \, \text{ if $1\le k\le l'$}\\
 		2\frac{\partial^2 J_{n_{k1},n_{k2}}}{\partial z_k^2} f_k f_k^\top+ 2 \frac{\partial^2 J_{n_{k1},n_{k2}}}{\partial \overline{z}_k^2}\overline{f}_k\overline{f}_k^\top + 4 \frac{\partial^2J_{n_{k1},n_{k2}}}{\partial z_k \partial \overline{z}_k}f_k \overline{f}_k^\top \text{ if } l'+1\le k \le l.
 	\end{dcases}
 \end{align*} 
By the normality of $\A$, the left eigenvectors of are orthonormal, so $\text{Tr}[\B \B^\top f_k f_k^*] = \langle \B^\top f_k, \B^\top f_k\rangle = \sigma_k^2$, and $\text{Tr}[\B \B^\top f_k f_j^*] = \langle \B^\top f_j, \B^\top f_k \rangle = \sigma_k\sigma_j \langle f_j,f_k\rangle = 0$. For cases where $f_k$ is complex, i.e., when $l'+1\le k \le l$, we also have $\text{Tr}[ \B \B^\top f_kf_k^\top] = \langle \B^\top \overline{f}_k, \B^\top f_k \rangle = \sigma_k^2 \langle \overline{f}_k,f_k\rangle = 0$.

Next observe that for $1\le k \le l'$, we have
\begin{align*}
	 \langle x,\A^\top\nabla \phi_{n_k}(x) \rangle + \frac{1}{2} \Tr \B\B^\top\nabla^2\phi_{n_k} (x) =& -\lambda_k \sqrt{\frac{2\lambda_k}{\sigma_k^2}}\langle x,f_k\rangle  \He_{n_k}'\left(\sqrt{\frac{2\lambda_k}{\sigma^2_k}}\langle x,e_k\rangle \right)\\& + \lambda_k \He_{n_k}''\left(\sqrt{\frac{2\lambda_k}{\sigma^2_k} } \langle x,e_k\rangle \right) \\ =& -n_k \lambda_k \He_{n_k}\left( \sqrt{\frac{2\lambda_k}{\sigma^2_k} } \langle x,e_k\rangle\right).
\end{align*}
For $l'+1 \le k \le l$,  we appeal to Proposition \ref{prop3} in the Appendix to obtain
\begin{align*}
	\langle x,\A^\top\nabla \phi_{n_k}(x) \rangle + \frac{1}{2} \Tr \B\B^\top\nabla^2\phi_{n_k} (x) = &-\sqrt{2}\overline{\lambda}_k \langle x, f_k \rangle\frac{\partial J_{n_{k1},n_{k2}}}{\partial z_k}  - \sqrt{2} \lambda_k \langle x, \overline{f}_k \rangle \frac{\partial J_{n_{k1},n_{k2}}}{\partial \overline{z}_k} 
	  \\&+ 2\sigma_k^2\frac{\partial^2J_{n_{k1},n_{k2}}}{\partial z \partial \overline{z}} \\
	  =& [-(n_{k1}+n_{k2})a_k + i(n_{k1}-n_{k2})b_k]J_{n_{k1},n_{k2}}.
\end{align*}
As for the cross terms, the normality of $\A$ implies that it is identically equal to zero. Therefore, we have $\mathcal{A} \phi_{\n}(x) = \mu_{\n}\phi_{\n}(x).$

 The above result also applies if $\B$ were a scalar multiple of an orthogonal matrix instead of being simultaneously diagonalizable with $\A$: i.e., when $\B = \sigma\mathbf{P}$ and $\mathbf{P}^\top\mathbf{P} = \mathbf{P}\mathbf{P}^\top = \mathbf{I}$.

\subsection{Applications of the special case eigenfunctions}
The eigenfunctions for the special cases above form complete orthonormal bases in $L^2(\nu)$, where $\nu$ is the invariant measure for the associated stochastic processes \cite{pavliotis2014stochastic,chen2014eigenfunctions}. The invariant density of $\nu$ is a normal distribution with mean zero and covariance 
  $\Sigma = \int_0^T e^{s\A}\B\B^\top e^{s\A^\top} \de s$ \cite{karatzas1998brownian}. Any function $g\in L^2(\nu)$ can then be expanded as an infinite sum of eigenfunctions, and the expansion coefficients can be expressed in terms of an integral with respect to the invariant measure: 
\begin{align}
 	g(x) = \sum_{\n} g_\n \phi_\n(x), \text{ where } g_\n = \int g(x) \phi_\n(x) \de \nu(x). 
 \end{align} 
The eigenfunctions of the $L^2(\nu)$-adjoint of the OU operator can also be found explicitly in this case. The adjoint operator is the Fokker--Planck operator of the stochastic process \cite{pavliotis2014stochastic}. The adjoint operator applied to a density $p\in L^2$ is
\begin{align}
	\sA^*p(x) = -\sum_{i = 1}^d \frac{\partial}{\partial x_i} \left[(\A x)_i p(x) \right] + \sum_{i,j = 1}^d \frac{\partial^2}{\partial x_i \partial x_j} \Q_{ij} p(x). 
\end{align}
The adjoint eigenfunctions are then $q_\n(x) = \phi_\n(x)p(x) $ with eigenvalue $\mu_\n$, where $p(x)$ is the invariant density. Solutions of the Kolmogorov backward equation (KBE) and Fokker--Planck equations can then also be expressed in terms of the eigenfunctions. For example, the KBE with terminal condition $g\in L^2(\nu)$: 
\begin{align*}
	\begin{dcases}
		\frac{\partial \Phi(t,x)}{\partial t} + \sA \Phi(t,x) &= 0 \\
		\Phi(T,x) &= g(x)
	\end{dcases}
\end{align*}
has solution
\begin{align*}
	\Phi(t,x) = \sum_{\n} g_\n e^{\mu_\n (T-t)}\phi_\n(x). 
\end{align*}
The solution of the Fokker--Planck equation can be obtained similarly. 

\section{Computation of general eigenfunctions}
\label{sec:general}
Here we turn to the case where we only assume $\A$ is diagonalizable. While in theory we know that the eigenfunctions can be expressed in closed form by polynomials, there is no simple way of expressing them in terms of classical orthogonal polynomials. Instead, we have found that a tractable approach for computing the eigenfunctions is to choose a basis of polynomials defined by the left eigenvectors of $\A$. Then, the action of the OU operator on the basis can be exactly represented by a matrix and the eigenfunctions are found by solving a matrix eigenvalue problem. We choose the basis $\{\psi_\n(x) \}_{\n \in \mathcal{I}}$, where the functions are defined in \eqref{eq:basisfuncts} and $\mathcal{I}\subset \mathbb{N}_0^d$ is some index set. This particular basis is chosen since its components are eigenfunctions of the first term of the OU operator. As we will see, this basis leads to a sparse matrix representation of the OU operator. Observe the following computation: 
\begin{align*}
    \mathcal{A} \psi_{\mathbf{n}}(x) &= \langle x,\A^\top \nabla\psi_{\mathbf{n}}\rangle + \Tr\left[ \mathbf{Q} \nabla^2 \psi_{\mathbf{n}}\right] \\& = \mu_{\mathbf{n}}\psi_{\mathbf{n}} + \Tr\left[ \mathbf{Q}\nabla^2 \psi_{\mathbf{n}}\right]. 
\end{align*}
We have that the trace term is
\begin{align*}
    \Tr\left[ \mathbf{Q}\nabla^2 \psi_{\mathbf{n}}\right]&= \sum_{k = 1}^d \Tr \left[\Q\nabla^2 \psi_{n_k}(x) \right] \prod_{j = 1,j \neq k}^d \psi_{n_j}(x) + 2\sum_{k = 1}^d \sum_{j = k+1}^d  \Tr\left[ \mathbf{Q}\nabla \psi_{n_k}\nabla \psi_{n_j}^\top \right] \prod_{l = 1,l\neq k, l\neq j}^d \psi_{n_l} \\
    & = \sum_{k = 1}^d \Tr\left[ \mathbf{Q}f_k f_k^\top\right] n_k(n_k-1) \langle x,f_k\rangle^{n_k-2}\prod_{j = 1, j \neq k}^d \psi_{n_j}(x) \\ &+ 2\sum_{k = 1}^d \sum_{j = k+1}^d \Tr\left[ \mathbf{Q} f_kf_j^\top\right] n_k n_j \langle x,f_k \rangle^{n_k-1} \langle x,f_j\rangle^{n_j-1} \prod_{l = 1,l\neq k, l\neq j}^d \psi_{n_l}.
\end{align*}
In more compact notation, we write
\begin{align}
	\sA \psi_{\n}(x)  = \mu_{\n}\psi_{\n}(x) + \sum_{k = 1}^d \langle \overline{f}_k ,\Q f_k \rangle n_k(n_k-1) \psi_{\m^{(k)}}(x) + 2\sum_{k = 1}^d \sum_{j = k+1}^d \langle \overline{f}_j,\Q f_k \rangle n_k n_j \psi_{\m^{(kj)}}(x)
	\label{lincomb}
\end{align}
where all entries of $\m^{(k)}$ and $\m^{(kj)}$ are equal to the corresponding entries of $\n$ except for $m^{(k)}_k = n_k -2$, and $m^{(kj)}_k = n_k-1$ and $m^{(kj)}_j = n_j -1$. Therefore, as long as $\m^{(k)}$ and $\m^{(kj)}$ are in $\mathcal{I}$, then $\sA\psi_{\n}(x)$ is contained in the span of $\{\psi_{\n}(x)\}_{\n\in \mathcal{I}}$. For practical computation, it is necessary to order the basis; lexicographical ordering is one obvious choice, but the choice is arbitrary and left to the user. Each basis function corresponds to an element of the standard basis, i.e., if there are $N = |\mathcal{I}|$ basis functions, then the $k$-th element of the basis corresponds to the vector in $\R^N$ with $1$ in the $k$-th entry and zero everywhere else. The matrix representation of $\sA$ is then $\mathbf{M} = [\sA \psi_{\n_1} \cdots \A \psi_{\n_R}].$

Suppose we are attempting to compute the eigenfunction with index $\n$. Based on \eqref{lincomb}, since the OU operator is a differential operator, $\mathcal{A}\psi_\n$ is itself a polynomial with index \emph{less than} $\n$ in the lexicographical ordering. This would require at most $N = \prod_{k = 1}^d (n_k+1)$ basis functions to span all the polynomials up to and including multi-index $\n$. The resulting matrix representation of $\mathbf{M}$ would then be an $N\times N$ matrix. However, \eqref{lincomb} implies that $\sA \psi_{\n}$ is dependent on at most $\frac{1}{2}(d^2+d+2)$ terms, which does not grow with the number of basis functions. Therefore, the resulting matrix is often quite sparse when many basis functions are considered. Solving the matrix eigenvalue problem would give \emph{all} of the eigenfunctions of $\sA$ with index up to and including $\n$.

Furthermore, if one only wishes to compute a single eigenfunction corresponding to index $\n$ (rather than all the eigenfunctions with total degree less than or equal to $\n$), then one does not need to include all the basis functions with index less than or equal to $\n$. For example, when $d = 2$ and we wish to compute the eigenfunction with index $(2,3)$, then the basis functions needed to express this eigenfunction have indices $\{(2,3),(2,1),(1,2),(1,0),(0,3),(0,1)\}.$ In Figure \ref{fig:sparsity}, we show the sparsity pattern of two matrix representations of the OU operator in high dimensions. Lexicographical ordering was used in constructing these matrices. Notice that in Figure~\ref{fig:a}, the matrix has size $1080 \times 1080$; in contrast, the matrix would be of size $2160 \times 2160$ if all indices less than or equal to $\n$ were included in the basis.
Similarly, in Figure~\ref{fig:b}, the matrix has size $17280 \times 17280$ rather than $34560 \times 34560$. The matrices were constructed by brute force, but they exhibit an interesting sparsity structure: for example, in  Figure~\ref{fig:b}, only $0.12\%$ of the matrix entries are nonzero. In future work, it may be interesting to investigate computationally efficient and structure-exploiting techniques for automatically constructing these matrices. 
\begin{figure}[h]
    \centering
    \begin{subfigure}{0.46\textwidth}
        \centering
        \includegraphics[width = \textwidth]{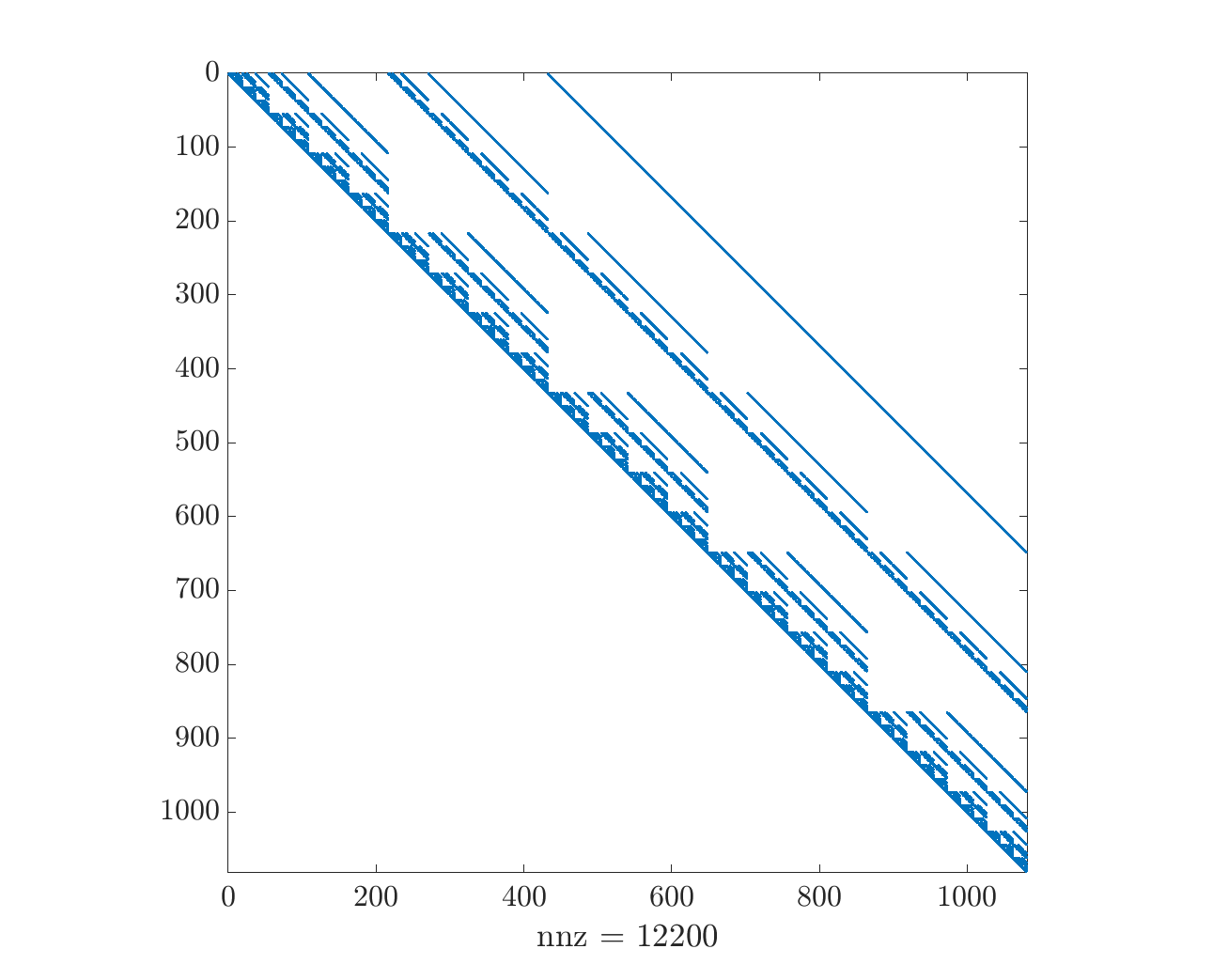}
        \caption{$d = 6$, $l' = 4$, $l = 5$, $\n = (4,3,2,2,2,3).$ The matrix is of size $1080\times 1080$.}
        \label{fig:a}
    \end{subfigure}%
    \hspace{10pt}
    \begin{subfigure}{0.46\textwidth}
        \centering
        \includegraphics[width = \textwidth]{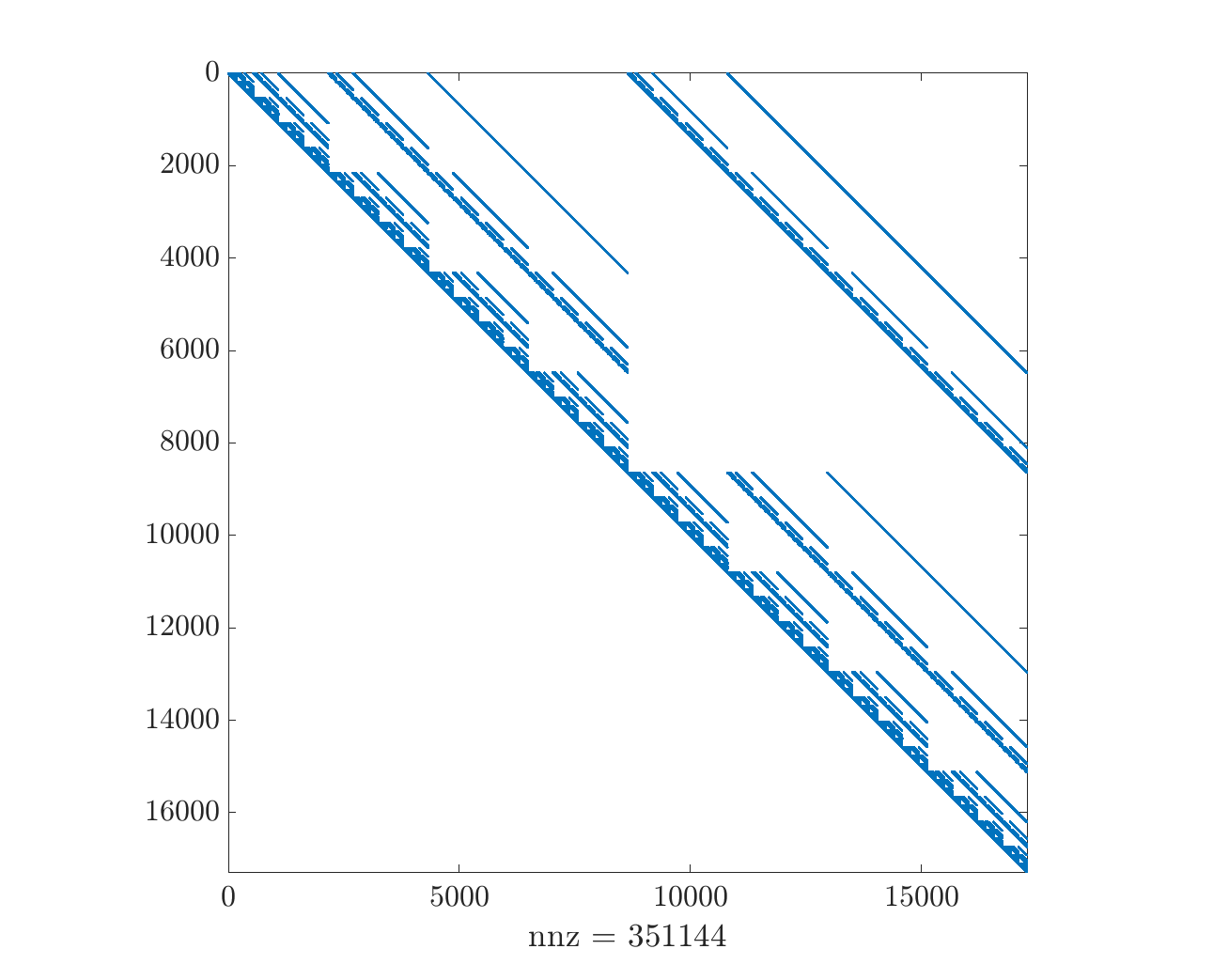}
        \caption{$d = 9$, $l' = 5$, $l = 7$, $\n = (1,3,3,2,2,1,3,4,2)$. The matrix is of size $17280 \times 17280$.}
        \label{fig:b}
    \end{subfigure}
    \caption{Sparsity patterns for two different matrix representations of $\sA$. }
    \label{fig:sparsity}
\end{figure}
As explored in \cite{leen2016eigenfunctions}, the eigenfunctions can be computed recursively via ladder operators. One could, therefore, express the ladder operators in terms of the basis we have chosen here, so that other eigenfunctions can be generated (given some initial high order eigenfunction). 

Lastly, we comment on numerical methods for solving this matrix eigenvalue problem. Recall that given an index $\n$, the corresponding eigenvalue $\mu_\n$ is known exactly by Proposition~\ref{prop1}, which means that only the eigenvectors need to be found. This means that only the nullspace of $\mathbf{M} - \mu_\n \mathbf{I}$ needs to be computed. In addition, $\mathbf{M}-\mu_\n\mathbf{I}$ is an \emph{upper triangular} matrix, which means that if $\mathbf{M}$ can be stored in memory (even in a sparse fashion), then the reduced row echelon form of the matrix can be easily computed and the nullspace can be found trivially. If only matrix-vector multiplies $\mathbf{M}v$ are accessible, the Arnoldi iteration can be employed to find the eigenvectors iteratively \cite{trefethen1997numerical}. 
 
\begin{remark}
One may ask if there is another choice of basis such that the number of terms produced by the trace term can be reduced. For example, a tempting choice is to use the basis defined in \eqref{eq:normaleigfunc}. We found that this choice yields a more complicated expression that is similar to \eqref{lincomb} without making the resulting matrix representation sparser.
\end{remark}

\begin{remark}
 Our approach is similar to that of \cite{rao2014jordan}, which computes eigenfunctions of the OU operator in the case that $\A$ is \emph{not} diagonalizable (in contrast with the present setting). More specifically, \cite{rao2014jordan} fixes a basis of polynomials (in fact, the tensorized Hermite polynomials) and seeks a finite-dimensional representation of the OU operator in that basis. However, eigenvalue problems of more than $d = 3$ dimensions were not studied. 
\end{remark}

\section{Conclusion}

We have presented a new 
approach for computing eigenfunctions of Ornstein--Uhlenbeck operators, in a general setting where the matrix $\A$ 
is diagonalizable.
We first collect results for special cases, e.g., when
$\A$ is self-adjoint or normal, and write explicit expresssions for the eigenfunctions in terms of certain orthogonal polynomials. We then address the general setting, where we show that by using a judicious choice of basis, one can compute eigenfunctions of any order, and in arbitrary dimension, by solving a sparse eigenvalue problem. The resulting matrix representation of the OU operator exhibits interesting structure that can be exploited to solve the associated eigenvalue problem efficiently.

These eigenfunctions have been found to be useful for applications such as simulating rare events \cite{zhang2021koopman} and approximating solutions to the Fokker--Planck equation \cite{leen2016eigenfunctions}. We anticipate that this approach will be relevant for many other applications.

\section*{Acknowledgments}
We thank Joshua White for contributing to the numerical experiments.

\bibliographystyle{unsrt}
\bibliography{bibliotheque.bib}

\appendix
\section{Hermite and Hermite-Laguerre-It\^o polynomials}
In this section we review the definitions of the Hermite and HLI polynomials, and some of their relevant properties.

\subsection{Hermite polynomials}
 There are many ways to define the probabilists' Hermite polynomials. The most relevant characterization for this note is the Hermite differential equation, which is an eigenvalue problem of the form
\begin{align}
	-x\phi_n'(x)+ \phi_n''(x) = \mu_n\phi_n(x).
\end{align}
The solutions to this differential equation are the Hermite polynomials $\phi_n(x) = \He_n(x)$ with eigenvalues $\mu_n = -n$ for $n\in\mathbb{N}_0.$ 

The Hermite polynomials (like any other univariate orthogonal polynomials) satisfy a three-term recurrence relation:
\begin{align}
	\He_{n+1}(x) = x\He_{n}(x)-n\He_{n-1}(x). 
\end{align}
Furthermore, derivatives of the Hermite polynomials can be expressed in terms of other, lower-order, Hermite polynomials as
\begin{align*}
	\frac{\de}{\de x}\He_n(x) = n\He_{n-1}(x). 
\end{align*}

\subsection{Hermite-Laguerre-It\^o polynomials}
The Hermite-Laguerre-It\^o (HLI) polynomials are bivariate orthogonal polynomials first studied by It\^o in his study of multiple complex-valued It\^o integrals. The definition of the HLI polynomials used in this note is from \cite{chen2014eigenfunctions}. A more comprehensive collection of the properties of these polynomials can also be found there. For integers $m,n$ and $(x,y)\in\R^2$, the polynomials are
\begin{align*}
	J_{m,n}(z,\overline{z}) = \begin{cases}
		(-1)^n n! z^{m-n} L_n^{m-n}(z\overline{z},\rho), \;\;\;\; m\ge n  \\
		(-1)^m m! \overline{z}^{n-m} L_m^{n-m}(z\overline{z},\rho),\;\;\;\; m<n
	\end{cases}
\end{align*}
where $z = x + iy$ and $L_k^\alpha(x,\rho)$ are the generalized Laguerre polynomials defined by the Rodrigues formula
\begin{align*}
	L_n^\alpha(x,\rho) = \frac{\rho^n}{n!}x^{-\alpha} e^{\frac{x}{\rho}}\frac{\de^n}{\de x^n} \left(e^{-\frac{x}{\rho}} x^{n+\alpha} \right),\;\;\;\; n\in\mathbb{N}. 
\end{align*}
The first six Hermite-Laguerre-It\^o polynomials for $\rho = 1$ are 
\begin{align*}
	&J_{0,0} = 1,\;	J_{1,0} = x+iy,\;J_{0,1} = x-iy \\
	&J_{1,1} = -(x^2+y^2)+1,\;J_{2,0} = (x+iy)^2,\;J_{0,2} = (x-iy)^2. 
\end{align*}
Like the Hermite polynomials, the derivatives of HLI polynomials can be written in terms of other HLI polynomials. Defining $z = x+iy$, we have,
\begin{align*}
	\frac{\partial}{\partial z} J_{m,n}(z,\bar{z}) = mJ_{m-1,n}(z,\bar{z}) \\
	\frac{\partial}{\partial \bar{z}}J_{m,n}(z,\bar{z}) = nJ_{m-1,n}(z,\bar{z}). 
\end{align*}
The following crucial result from \cite{chen2014eigenfunctions} shows that $J_{m,n}$ are the OU eigenfunctions
\begin{proposition}[\cite{chen2014eigenfunctions}, Theorem 2.6]
	The Hermite-Laguerre-It\^o polynomials satisfies
	\begin{align*}
		\left[ \overline{\lambda} z\frac{\partial }{\partial z} + {\lambda}\overline{z}\frac{\partial}{\partial \overline{z}} +2 \sigma^2 \frac{\partial}{\partial z\partial \overline{z}} \right] J_{m,n}(z,\bar{z}; \rho) = \mu_{m,n}J_{m,n}(z,\bar{z}; \rho)
	\end{align*}
	where $\lambda = -a + ib$, $\rho = \sigma^2/a$, and $\mu_{m,n} = -(m+n)a+i(m-n)b.$ \label{prop3}
\end{proposition}

\end{document}